\documentclass{amsart}
\usepackage{mathrsfs}
\usepackage[all]{xy}
\usepackage{mathabx} 

\usepackage{wasysym}
\usepackage{amssymb}
\usepackage{comment}
\usepackage{mathtools}
\usepackage{mathbbol}

\usepackage{ifpdf}
\ifpdf 
  \usepackage[pdftex]{graphicx}
  \DeclareGraphicsExtensions{.pdf,.png,.jpg,.jpeg,.mps}
  \usepackage{pgf}
\else 
  \usepackage{graphicx}
  \DeclareGraphicsExtensions{.eps,.bmp}
  \usepackage{pgf}
  \usepackage{pstricks}
\fi
\usepackage{epic,bez123}
\usepackage{wrapfig}

\DeclareMathAlphabet{\mathpzc}{OT1}{pzc}{m}{it}

\newtheorem{thm}{Theorem}[section]

\newtheorem{cor}[thm]{Corollary}
\newtheorem{lem}[thm]{Lemma}

\newtheorem*{claim}{Claim}

\theoremstyle{definition}
\newtheorem{defn}[thm]{Definition}

\theoremstyle{remark}

\newtheorem*{rem}{Remark}

\newcommand{\GP }{(G, \mathcal P)}

\newcommand{\pX}{\partial{X}}

\newcommand{\diam }[1]{{\textbf{diam}\big(#1\big)}}

\usepackage[bookmarks=true, pdfauthor={YANG Wenyuan}]{hyperref}

\begin{document}

\title[Lower bound on the growth rate]{Lower bound on  growth of non-elementary subgroups in relatively hyperbolic groups}

\author{Yu-miao Cui}
\address{School of Mathematics, Hunan University,Changsha, Hunan, 410082, China}
\email{cuiyumiao@hnu.edu.cn}
\author{Yue-ping Jiang}
\address{School of Mathematics, Hunan University,Changsha, Hunan, 410082, China}
\email{ypjiang@hnu.edu.cn}

\author{Wen-yuan Yang}
\address{Beijing International Center for Mathematical Research (BICMR), Peking University, No. 5 Yiheyuan Road, Haidian District, Beijing, China}
\email{yabziz@gmail.com}

\thanks{Y-P. J. is supported by the National Natural Science Foundation of China (No. 11631010). W-Y. Y. is supported by the National Natural Science Foundation of China (No. 11771022).}



\date{\today}

\dedicatory{}


\begin{abstract}
This paper proves that in a non-elementary relatively hyperbolic group,  the logarithm  growth rate of any non-elementary subgroup has a  linear lower bound by the logarithm of the size of the corresponding generating set. As a consequence, any non-elementary subgroup has uniform exponential growth.   
\end{abstract}

\maketitle
\section{Introduction}
 
\subsection{Results and background} Let $S$ be a finite symmetric generating set of a group $H$ and $d_S$ the corresponding word metric.  Denote   $$\forall n\in \mathbb N\cup\{0\},\; S^{\le n}:=\{h\in H: d_S(1,h)\le n\}.$$
The (logarithm) \textit{growth rate} of $H$ with respect to $S$ is defined to the following limit $$\omega(H,S):=\lim_{n\to\infty}\frac{\log{\sharp (S^{\le n})}}{n}$$
which exists since $\sharp (S^{\le {n+m}})\le \sharp (S^{\le n})\cdot \sharp (S^{\le m})$. In what follows, we always consider finitely generated groups. 


The spectrum of growth rates of a group $H$ has attracted lots of research interests:
$$
\Omega(H):=\displaystyle  \left\{ {\omega(H, S)}: \sharp S<\infty,\; \langle S\rangle =H  \right \}.$$
For a group with exponential growth, the question of Gromov \cite{Gro81} whether $\Omega(H)$ admits the infimum $0$  was open for twenty years and  answered negatively by Wilson \cite{Wi04} (see \cite{Bar03} also). He constructed the first examples of groups with {non-uniform exponential growth} so that a sequence of two-element generating sets with growth rates tending to 0. 

A   group $H$ has   \textit{uniform exponential growth} if $\inf \Omega(H)>0$. If a group has uniform exponential growth, it is quite interesting to ask whether $\Omega(H)$ obtains the minimum. Sambusetti \cite{Sam2} showed that the answer was again negative for the free products of any two non-Hopfian groups which are a special class of relatively hyperbolic groups. However, the recent  work by Fujiwara-Sela \cite{FS20} obtains a positive answer for hyperbolic groups by showing the set $\Omega(H)$ is well-ordered, so $\Omega(H)$ admits a minimum.  This settles a question of de la Harpe.  The starting point of their arguments relies on the fact due to Arzhantseva-Lysenok \cite{AL06} that the growth rate $\omega(H,S)$ is lower bounded by a linear function of  the size of   $S$.

Noting the simple fact $\omega(H,S)\le \log \sharp S$,   the work of \cite{FS20} and \cite{AL06} seem to suggest worthy   understanding  the following set
$$
\Theta(H):=\displaystyle  \left\{ \frac{\omega(H, S)}{\log \sharp S}: \sharp S<\infty,\; \langle S\rangle =H  \right \}.$$
Of course, $\Theta(H)\subset [0, 1]$. A number of inquiries could be made about the nature of  $\Theta(H)$. For instance, could the set  $\Theta(H)$ always be infinite? If it is infinite,   what are the accumulation points of the set $\Theta(H)$?   The purpose of this paper is not to give complete answers to these questions. Instead, we collect here a few simple observations to motivate further investigations. 

A group $H$ has \textit{purely exponential growth} if $\frac{1}{C} \exp(n\omega) \le \sharp S^n\le  C\exp(n\omega)$ for some $C>0$ independent of $n\ge 1$. This class of groups includes (relatively) hyperbolic groups and many other groups (see \cite{Coor} and \cite{YANG10} for relevant discussions). By taking $T_n:=S^n$, one sees \footnote{The authors learnt this fact from Alex Furman.} that $$\frac{\omega(H, T_n)}{\log \sharp T_n}\to 1,\;\text{as } n\to\infty $$ Thus, the upper bound $1$ is an accumulation point for any group with purely exponential growth. On the other hand, the growth tightness \cite{GriH} of free groups implies that  $1\in \Theta(H)$ if and only if $H$ is a free group. Thus it is interesting to ask whether there exist examples with  $\Theta(H)\subset [0, 1-\epsilon]$ for some $\epsilon>0$.  

The examples of Wilson also imply $0\in \overline {\Omega(H)}$ for certain non-uniform exponential growth groups $H$. Analogous to  the question of uniform exponential growth, we can ask for which   groups $\Theta(H)$ admits a positive infimum. In fact, $\inf \Theta(H)>0$ has been obtained for hyperbolic groups in \cite{AL06}.

The main result of this paper is a generalization of  the previous results of Arzhantseva-Lysenok \cite{AL06} to   the class of  relatively hyperbolic groups. Since the official introduction in the Gromov 1987 monograph \cite{Gro}, this class of  groups has been well-studied in last thirty years, see \cite{Farb}, \cite{Bow1}, \cite{Osin}, \cite{DruSapir}, \cite{Ger}. The important examples include Gromov-hyperbolic groups, geometrically finite Kleinian groups (with variable negative curvature), infinitely-ended groups,  small cancellation quotients of free products, limit groups, to name just a few. 

Our main theorem establishes the positive lower bound on $\Theta (H)$ for any non-elementary subgroup in a relatively hyperbolic group. By definition, a subgroup $H$ is called \textit{non-elementary} if its limit set contains at least 3 points. See \textsection\ref{ElemSection} for details.

\begin{thm}\label{mainthm2}
Assume that  $G$ is a non-elementary relatively hyperbolic group. Then there exists a constant $\kappa=\kappa(G)\in (0, 1]$ such that for any non-elementary subgroup $H$ with a finite symmetric generating set $S$, we have 
$$
\omega(H, S) \ge \kappa \cdot \log{ \sharp S}.
$$  
\end{thm} 
\begin{rem}
Wilson's example exhibits a sequence of 2-generator sets with growth rate tending 0. This shows that the non-elementary assumption of $H$ is necessary: indeed, any group $H$ can be realized as the maximal parabolic subgroup in  a free product of $H$ with any nontrivial group. 
\end{rem}


A group $G$ has \textit{uniform uniform exponential growth} if every finitely generated subgroup has uniform exponential growth. Xie \cite{Xie} has proved that   relatively hyperbolic groups have uniform exponential growth. As a direct corollary of Theorem \ref{mainthm2}, we obtain a strengthening of Xie's theorem.
\begin{thm} \label{mainthm}
Any non-elementary subgroup $H$ of a non-elementary relatively hyperbolic group $G$ has uniform exponential growth.  
\end{thm} 
\begin{rem}
Recall that non-elementary relatively hyperbolic groups $G$ are growth tight, so non-Hopfian ones cannot realize its infimum of $\Omega(G)$ (see \cite{YANG6} \cite{Sam2}). It is thus interesting to known whether  Fujiwara-Sela's result \cite{FS20} can generalize to torsion-free toral relatively hyperbolic groups \cite{G05,G09}. \footnote{This question has been answered for \textit{equationally Noetherian} relatively hyperbolic groups  by Fujiwara's preprint \cite{F21} (posted on 2 Mar. 2021, one day earlier  than ours on arXiv). One ingredient is Theorem \ref{mainthm} which was also obtained by him independently and simultaneously.} 
\end{rem}

\subsection{Connection with other works}

It has been recent interests  to study the product set growth in various classes of groups, starting in free groups \cite{Saf11}, hyperbolic groups and acylindrical hyperbolic groups \cite{DS20}, free product of groups \cite{B13}, and so on.   We refer the reader to \cite{B13} for   further references and connection with approximate groups.  

To be precise, let $S$ be any set in a group $G$ subject to the condition $S$ do not generate a ``small" subgroup. The   {H}elfgott type growth (in the terminology of \cite{B13}) wishes to have the following $$\sharp (S^3) \ge  c \cdot (\sharp S)^{1+\kappa}$$
for some universal $c, \kappa>0$ depending only on $G$. By induction, it is easy to see that if a group $G$ has the {H}elfgott type growth, then Theorem \ref{mainthm2} holds for this group $G$. In this sense,   Theorem \ref{mainthm2} could be understood as asymptotic version of product set growth. Indeed, our proof boils down to a similar product growth with high powers
$$
\forall i\in \mathbb N,\; \sharp (S^{i\kappa}) \ge   (\sharp S)^i
$$
for a universal $\kappa>0$. Even though, our Theorem \ref{mainthm2} cannot be deduced directly from the  result \cite[Theorem 1.9]{DS20}. Their result does provide certain product set growth   only assuming  the acylindrical action on hyperbolic spaces. However,   the large displacement assumption imposed there on $S$ is hard to verify in practice.   
 
Very recently,    Kropholler-Lyman-Ng  obtained independently Theorem \ref{mainthm} as \cite[Proposition 4.12]{KLN} during our writing of this paper. Similar to us, they made a variant of Xie's result  as Lemma \ref{LargeDiamLem}, and then run the remaining argument in \cite{Xie} to get Theorem \ref{mainthm}.   

To conclude the introduction, let us mention briefly the proof of main theorems. We follow   closely the strategy of \cite{AL06} which appears to us quite robust. On the other hand, we have to deal with several difficulties from the relative case. They are resolved largely by adapting the work of Xie \cite{Xie} (see Lemma \ref{LargeDiamLem}) and by a strengthening of Koubi's result \cite{Koubi} (see Lemma \ref{ShortHypLem}). We believe that Lemma \ref{ShortHypLem} has independent interest and admits further applications.

\textbf{Structure of the paper.} This paper is organized as follows. Section \ref{PrelimSection} recalls standard materials in Gromov's hyperbolic geometry, Bowditch-Gromov's definition of relatively hyperbolic groups. As mentioned above, the work of Xie and Koubi are  properly adapted and strengthened in  Section \ref{ShortLoxSection}. A notion  of loxodromic elements with large injectivity is introduced in Section \ref{LargeInjSection} to streamline  the strategy of Arzhantseva-Lysenok. The proof of Theorem \ref{mainthm2} is then completed in Section \ref{MainProofSection}. 

\textbf{Acknowledgment.}
We would like to thank  Igor Lysenok   for helpful conversations and Thomas Ng for several corrections.

\section{Preliminary}\label{PrelimSection}
Consider an isometric action of $G$ on a metric space $(X,d)$.  Let $S\subset G$ be a set of isometries. 
Denote $\ell_x(S):= \max_{s\in S}\{d(x,sx)\}$ for a given point $x\in X$. For a subset $A\subset X$, define
$$\ell_A(S):=\displaystyle \inf_{x\in A} \ell_x(S).$$ 
Note that $x\in X\mapsto \ell_x(S)\in \mathbb R$ is a continuous non-negative function. 

\subsection{Hyperbolic spaces and Loxodromic elements}

Define the Gromov product $$\forall x,y,o\in X,\; \langle x, y\rangle_o=\frac{d(x,o)+d(y,o)-d(x,y)}{2}.$$
A geodesic metric space $X$ is called \textit{hyperbolic} if any geodesic triangle is $\delta$-thin: if $d(o, p)=d(o,q)\le \langle x, y\rangle_o$ for two points $p\in [o,x], q\in [o,y]$, then $d(p,q)\le \delta$.
Then for any $x, y,z, o\in X$, we have
$$
\langle x,y\rangle_o \ge \min\{\langle x,z\rangle_o, \langle z,y\rangle_o\}-\delta.
$$

Assume that a finitely generated group $G$ acts properly by isometry on a proper hyperbolic space $X$. Then the induced action of $G$ on the Gromov boundary $\partial X$ of $X$ is a convergence group action. Thus, any infinite order element fixes at least one but at most two points in $\partial X$. So  the elements in $G$ are classified into three nonexclusive classes: \textit{elliptic isometry} with finite order elements, \textit{parabolic isometry} with only fixed point and \textit{loxodromic isometry} with exactly two fixed points. See \cite{Bow1} for a detailed discussion about convergence group actions and relevant notions.

Equivalently, an isometry $g$ on a proper hyperbolic space $X$ is   {loxodromic} if it admits  a $(\lambda,c)$-quasi-geodesic $\gamma$ for some $\lambda,c>0$ so that  $\gamma, g\gamma$ have finite Hausdorff distance. Such quasi-geodesics shall be referred to as  $(\lambda,c)$-\textit{quasi-axis}.

\begin{lem}\cite[Lemma 9.2.2]{CDP}\label{LoxoCriterionLem}
If $g$ is an isometry satisfying  $$d(o,go) \ge 2\langle o, g^2 o\rangle_{go} + 6\delta$$ for some point $o \in X$, then $g$ is loxodromic.  
\end{lem}

\begin{lem}\cite[Lemma 1]{AL06}\label{LocalGeodLem}
Let $x_1, x_2, \cdots, x_k$ for $k\ge 3$ be points in a $\delta$-hyperbolic space such that for any $2\le i\le k-2$, we have
$$
\langle x_{i-1},x_{i+1}\rangle_{x_i}+\langle x_{i},x_{i+2}\rangle_{x_{i+1}}\le d(x_i, x_{i+1})-3\delta
$$
Then 
$$
d(x_1,x_k)\ge \sum_{i=1}^{k-1} d(x_i, x_{i+1}) -2\sum_{i=2}^{k-1} (\langle x_{i-1},x_{i+1}\rangle_{x_i} +\delta ).
$$
\end{lem}
The following immediate corollary will be actually used.  
\begin{cor}\label{LocalGeodCor}
Under the assumption of Lemma \ref{LocalGeodLem}, if 
$$
\langle x_{i-1},x_{i+1}\rangle_{x_i}+\langle x_{i},x_{i+2}\rangle_{x_{i+1}}\le d(x_i, x_{i+1})/4-\delta
$$
then
$$
d(x_1,x_k)\ge \frac{1}{2} \sum_{i=1}^{k-1} d(x_i, x_{i+1}).
$$
\end{cor}

\begin{lem}\label{LoxoCriterion2Lem}
If $g, h$ are two  isometries satisfying  $$\frac{1}{4}\min\{d(go, o), d(ho, o)\} \ge  \max\{\langle go, h^{-1} o\rangle_o,\langle g^{-1}o, h o\rangle_o\} + \delta$$ for some point $o \in X$. Then 
\begin{enumerate}
    \item 
    $gh$ is loxodromic.
    \item
    there exist constants $\lambda, c>0$ depending only on $\delta$ such that the concatenated path $\bigcup_{i\in\mathbb Z} (gh)^{i}\left([o, go]\cdot g[o,ho]\right)$ is a $(\lambda,c)$-quasi-geodesic. 
    \item
    there exists a constant $C=C(\delta)>0$ such that 
    $|\ell_X(gh) -d(o, gho)|\le C.$
\end{enumerate}  
\end{lem}
\begin{proof}
The proof uses the well-known fact that long local geodesics are global quasi-geodesics. To be precise, applying Lemma \ref{LocalGeodLem} to the points $x_1=o,\; x_2=go, \;x_3=gho,\;...,\; x_{2n+1}=(gh)^{n}o$,    we have
$$
d(x_1, x_{2n+1}) \ge  \sum_{k=1}^{2n} d(x_i, x_{i+1})-2\sum_{k=2}^{2n}  (\langle x_{i-1}, x_{i+1}\rangle_{x_i} +\delta )\ge \frac{1}{2}\sum_{k=1}^{2n} d(x_i, x_{i+1}). 
$$
The proof is completed.
\end{proof}

Define the asymptotic \textit{translation length} of an isometry $g$ as follows
$$
\tau(g):=\lim_{n\to\infty}\frac{d(o, g^no)}{n}
$$
for some (thus any) point $o\in X$.
\begin{lem}\cite[Proposition 10. 6.4]{CDP}\label{TransLengthLem}
If $g$ is a loxodromic element, then $|\ell_X(g)-\tau(g)|\le 16\delta.$
\end{lem}

We say that an element $g\in G$ \textit{preserves the orientation} of the bi-infinite quasi-geodesic $\gamma$ if $\alpha, g\alpha$ has finite Hausdorff distance for any half-ray $\alpha$ of  $\gamma$. It is clear that a loxodromic element preserves the orientation of any quasi-axis.

\begin{lem}\label{OrientationLem}
If $g$ preserves the orientation of $(\lambda, c)$-quasi-axis $\gamma$, then there exists a constant $C$ depending on $\lambda, c, \delta$ with the following property. For any $x\in \gamma$, there exists $ y\in \gamma$ such that    $\langle x, gx\rangle_y<C$ and $\langle y, gy\rangle_{gx}<C$.
\end{lem}

The following lemma is well-known with  proof included for completeness.
\begin{lem}\label{QuasiAxisLengthLem} 
There exists a constant $C=C(\lambda,c,\delta)$ for any $\lambda,c>0$ with the following property. If  a loxodromic element $g$ admits a  $(\lambda, c)$-quasi-axis $\gamma$, then for any $x\in \gamma$, we have $|\ell_X(g)-d(x, gx)|\le C$.
\end{lem}
\begin{proof}
By Morse Lemma, any two $(\lambda, c)$-quasi-axes $\gamma, g\gamma$ have bounded Hausdorff distance depending only on $\lambda, c, \delta$. Thus, the inclusion of $\gamma$ into $\cup_{i\in \mathbb Z} g^i\gamma$ is a quasi-isometry with constants depending $\lambda, c, \delta$ only. We can thus assume that the quasi-axis $\gamma$ is $\langle g\rangle$-invariant. 

Note that the shortest projection $\pi_\gamma(\cdot)$ to a $(\lambda,c)$-quasi-geodesic $\gamma$ is $C$-contracting for a constant $C=C(\lambda, c, \delta)$:  
$$(\forall z,w \in X,\; diam(\{\pi_\gamma(z),\pi_\gamma(w)\})>C) \Longrightarrow (\max\{d(\pi_\gamma(z),\gamma), d(\pi_\gamma(w),\gamma)\}\le C).$$

We then derive the following for any $z,w\in X$: $$diam(\{\pi_\gamma(z),\pi_{g\gamma}(w)\})+d(z, \pi_\gamma(z))+d(w,\pi_\gamma(w)) \le d(z,w)+4C.$$  
Let $o\in X$ so that $d(o, go)=\ell_X(g)$. We apply the above inequality for $z=o, w=go$. Using the fact that $\gamma$ is $\langle g\rangle$-invariant, we obtain $d(o, \gamma)\le 2C$: indeed, if not,  we would have $d(\pi_\gamma(o), g\pi_\gamma(o))<d(o, go)$. Thus, by taking its projection $\pi_\gamma(o)$, it suffices to prove the conclusion by assuming that  $o$ lies on $\gamma.$

Let $x\in \gamma$. By Lemma \ref{OrientationLem}, we can assume up to translation  that $\langle o, go\rangle_x\le C$ and then $\langle x, gx\rangle_{go}\le C$. Thus, $d(x, [o,go])\le C$ and $d(go, [x,gx])\le C$. Consequently,   
\begin{align*}
|d(o, go)-d(x, gx)| &\le |d(o,x)+d(x,go)-d(x,go)-d(go,gx)|+4C\\
&\le 4C.
\end{align*}
The proof is complete.
\end{proof}

\subsection{Elementary subgroups}\label{ElemSection}

Recall that $G$ acts properly on a proper hyperbolic space $X$. The \textit{limit set $\Lambda H$} of a subgroup $H$ is the set of accumulation points in $\partial X$ of any $H$-orbit in $X$. A subgroup $H$ in $G$ is called \textit{elementary} if its limit set contains at most two points. See \cite{Bow1} for relevant discussion.

If $\Lambda H$ consists of only one point $p$, then $H$ is called \textit{parabolic subgroup} and $p$ is called a \textit{parabolic point}. It is a well-known fact that in a convergence group action, a loxodromic element cannot   fix a parabolic point. The (maximal) parabolic group plays the key role in definition \ref{RHdefn} of relatively hyperbolic groups given in the next subsection. In the remainder of this subsection, we first consider the elementary subgroup with exactly two limit points.


Let $\gamma$ be a quasi-axis for a loxodromic element $h$.  The coarse stabilizer of the axis defined as follows
$$E(h)=\{g\in G: \exists r>0, \gamma\subset N_r(g\gamma), g\gamma\subset N_r(\gamma)\}$$
gives the maximal elementary subgroup containing $h$.
Note that the following index at most 2 subgroup   $$E^+(h):=\{g\in G:  \exists n > 0,\;  gh^ng^{-1}=h^n\}$$ is precisely the set of orientation-preserving elements in $E(h)$. 
 
Denote $E^-(h)=E(h)\setminus E^+(h)$.
Let $E^\star(h)$ be the torsion group of $E^+(h)$.

\begin{lem}\label{elementarygroup}
 For a loxodromic element $h$,  the following statements hold:
\begin{enumerate}
\item
$[E(h): \langle h \rangle] <\infty$, and $E(h)$ is a contracting subgroup with bounded intersection. 
\item
$
E(h)=\{g\in G: \exists n > 0,\; (gh^ng^{-1}=h^n)\; \lor\;  (gh^ng^{-1}=h^{-n})\}.
$
\item
$E^\star(h)$ is a finite normal subgroup of $E(g)$.
\item
$g^2\in E^\star(h)$ for any $g\in E^-(h)$.
\end{enumerate}
\end{lem}
\begin{proof}
The first two statements are \cite[Lemma 2.11]{YANG10}.
The last two statements follow from \cite[Lemma 4]{AL06} where only the assertion (1) is used in the proof.
\end{proof}

\begin{lem}\label{ShortEllipticLem}
Let $h\in G$ be a loxodromic element admiting a $(\lambda,c)$-quasi-axis $\alpha$. Then there exists $D=D(\lambda,c, \delta)$ such that $d(go, o)\le D$ for any $g\in E^\star(h)$ and $o\in \alpha$.
\end{lem}
\begin{proof}
By Morse Lemma, there exists a constant $D>3\delta$ depending only on $\lambda, c, \delta$ such that any two of $g\alpha, \alpha, g^{-1}\alpha$ have Hausdroff distance at most $D$. Since $g\in E^\star(h)\subset E^+(h)$ preserves the orientation of $\alpha$, up to taking inverse of $g$, we can assume further that $d(go, [o,\alpha_+]_\alpha)\le D$ and $d(g^{-1}o, [o,\alpha_-]_\alpha)\le D$. Let $x, y\in \alpha$ such that $d(go, x), d(g^{-1}o, y)\le D$. Since $x, y$ are on the opposite sides of $o$ on $\alpha$, we have $d(o, [x,y])\le D$. Thus, $\langle x, y\rangle_o\le D$ and then $\langle go, g^{-1}o\rangle_o\le 3D$.  If $d(o,go)>4D$ was assumed,  then $d(o,go)\ge   2d(o, g^2o)+6\delta.$ By Lemma \ref{LoxoCriterionLem}, $g$ is loxodromic. This is a contradiction, so $d(o,go)\le 4D$.
\end{proof}

\subsection{Relatively hyperbolic groups}
The notion of  a relatively hyperbolicity has a number of equivalent formulation (see \cite{Farb}, \cite{Bow1}, \cite{Osin}, \cite{DruSapir}, \cite{Ger} etc). See \cite{Hru} for a survey of their equivalence.  In this paper we define a relatively hyperbolic group which admits a cusp-uniform action on a hyperbolic space.
\begin{defn}\label{RHdefn}
Suppose $G$ admits a proper and isometric action on a proper
hyperbolic space $(X, d)$ such that $G$ does not fix a point in the Gromov boundary $\pX$. Denote by $\mathcal P$ the set of
maximal parabolic subgroups in $G$. Assume that there is a
$G$-invariant system of disjoint (open) horoballs $\mathbb U$ centered at
parabolic points of $G$ such that the action of $G$ on the complement called \textit{neutered space}
$$X(\mathbb U):=X \setminus \mathcal U$$
is co-compact where $\mathcal U := \bigcup_{U \in \mathbb U} U$. Then the pair $\GP$ is said to be \textit{relatively
hyperbolic}, and the action of $G$ on $X$ is called
\textit{cusp-uniform}.
\end{defn}

We fix a $G$-invariant system $\mathbb U$ of horoballs and a neutered space $X(\mathbb U)$ on which $G$ acts co-compactly. The following result is proved by \cite[Lemma 6]{AL06} in hyperbolic groups. In the relative case, we follow closely their arguments.

\begin{lem}\label{LinearBILem}
Let $h$ be a loxodromic element in $G$ so that for some point $o\in X$ and $\lambda, c>0$,  the path  $\alpha=\cup_{n\in \mathbb Z} [h^no, h^{n+1}o]$ is a $(\lambda, c)$-quasi-geodesic  in $X$.  Then for any given $\theta>0$  there exists $N=N(\lambda, c,\delta,  \theta)$ independent of the point $o$ such that for any $f\notin E(h)$,  we have $$\diam{\alpha\cap N_R(f\alpha)}\le N\cdot  d(o, ho)$$
where $R:=\theta\cdot d(o, ho).$
\end{lem}
\begin{proof}
First of all, since $h$ is a loxodromic element and cannot fix any parabolic point, we obtain that $\alpha$ cannot be contained inside any horoball $U\in\mathbb U$. Thus, the $\langle h\rangle $-invariant set $\alpha\cap X(\mathbb U)$ is a non-empty unbounded set. Namely, for any $x\in \alpha\cap X(\mathbb U)$ and any $i\in \mathbb Z$,  we have  $h^ix\in \alpha$.

We argue by contradiction.  Assume that $\diam{\alpha\cap N_R(f\alpha)}> N\cdot  d(o, ho)$ for a constant $N$ determined below.  Let $z, w,z',w'\in \alpha$ such that $d(z,w)=\diam{\alpha\cap N_R(f\alpha)}$ and $d(z,fz'),\; d(w,fw')\le R$.

By hyperbolicity, $[z,w]_\alpha$ and $[z',w']_{\alpha}$ contain subpaths $\beta_1, \beta_2$ respectively such that $\beta_1, f\beta_2$ have Hausdorff distance at most $C=C(\lambda, c, \delta)>0$ and for $i=1,2$, we have  $$\diam{\beta_i}\ge \diam{\alpha\cap N_R(f\alpha)}-2R\ge (N-\theta)d(o,ho).$$ Since $\alpha$ is a $(\lambda, c)$-quasi-geodesic, there exists a monotone increasing function $N'=N'(\lambda, c, N, \theta)>0$ such that $\beta_1$ contains at least $(N'+1)$ translates of $[o,ho]$.  Moreover, $N'=N'(\lambda, c, N, \theta)\to\infty$ as $N\to\infty$. Thus, $\beta_1$ contains $(N'+1)$ points $x, hx, \cdots, h^{N'}x\in X(\mathbb U)$.  Let $y\in \beta_2$ be a point so that $d(x,fy )\le C$.

Assume that $C$ also satisfies the conclusion of Lemma \ref{QuasiAxisLengthLem}. With Lemma \ref{TransLengthLem},  for $1\le i\le N'$,  we have $$|d(x, h^ix) - \tau(h^i)|\le C+16\delta,\quad |d(fy, fh^iy) -\tau(h^i)|\le C+16\delta$$ so  $$|d(x, h^ix)-d(fy, fh^i y)| \le 2C+32\delta$$ Thus,   $d(h^ix, fh^{i}y)\le 3C+32\delta$ for each $1\le i\le N'$. 
 

Set $N(x,y)=\sharp \{g\in G: d(x, gy)\le C\}+1$. Since $G$ acts cocompactly on the $C$-neighborhood of $X(\mathbb U)$, we see that $N(x,y)$  over $x, y\in N_C(X(\mathbb U))$ is uniformly   bounded above by a constant $N_0$. 

Choose $N>0$ such that  $N'=N'(\lambda, c, N, \theta) \ge N_0$, and consequently, we obtain $h^{-i}fh^{i}=h^{-j}fh^{j}$ for $1\le i\ne j\le N'$. So   $f\in E(h)$  contradicts with the assumption. The result is proved.
\end{proof}

At last, let us mention the following result of Osin which holds for loxodromic elements in any acylindrical action on hyperbolic spaces.
 
\begin{lem}\cite[Lemma 6.8]{Osin6}\label{UniFiniteBoundLem}
There exists a finite number $N_0$ such that $\sharp E^\star(g)\le N_0$ for any loxodromic element $g\in G$.
\end{lem}



\section{Short loxodromic elements}\label{ShortLoxSection}
The goal of this section is to provide  short loxodromic elements.

Let $S=S^{-1}$ be a symmetric generating set of a non-elementary group $H$. Recall   $S^{\le n_0}:=\{h\in H: d_S(1,h)\le n_0\}$.

The following is a variant of \cite[Lemma 5.3]{Xie}.
\begin{lem}\label{LargeDiamLem}
For any $M>0$, there exists a positive integer $n_0=n_0(M)>0$ such that for any  finite symmetric generating set $S$ of $H$, we have
$$
\ell_X(S^{\le n_0}) >M.
$$  
\end{lem}
\begin{proof}
Let $\mathbb U$ be a $M$-separated $G$-invariant system of horoballs centered at the parabolic points. Recall that the action of $G$ on $X(\mathbb U)$ is proper and co-compact. Let $K \subset X(\mathbb U)$ be a compact set such that $\cup_{g\in G} g(K) = X(\mathbb U)$. Fix a point $p\in K$ and denote $a = diam(K)$ depending on $M$. The proper action implies the set  $$A=\{g\in G:d(g(p),p)\le 2a+M\}$$ is a finite set. Since $G$ is finitely generated,  up to increasing the value of $a$, we can assume that  $A$ generates $G$. 

Consider the finite set $\mathbb H$ of conjugates of $H$ which is generated by some finite set $S'\subset A$.  Since $H$ is infinite, the proper action of $H$ on $X$ implies that for every $H'\in\mathbb H$, there is some $g_{H'} \in H'$ with $d(g_{H'} (p), p) > M + 2a$. Since $A$ generates $G$ and $\mathbb H$ is finite, then the integer 
$$n_0 := \max\{d_{S}(1,g_{H'}) : S' \subset A, H':=\langle S'\rangle \in \mathbb H\}$$
is finite.

Now let $S$ be a finite generating set of $H$. If $\ell_X(S) > M$, then we are done: $\ell_X(S^{n_0})\ge \ell_X(S)>M$. If there is some $x \in X$ with
$\ell_x(S) \le M$, then    $x\in X(\mathbb U)$. Indeed, assume that  $x\in U$ for some $U\in \mathbb U$. By definition of $\ell_x(S) \le M$ we have $d(s(x), x) \le M$ for all $s\in S$. The $M$-separation of $\mathbb U$ implies $s(U)=U$ for all $s \in S$ and so the center of $U$ would be fixed by the non-elementary subgroup $H$: a contradiction. Hence, it follows that $x\in X(\mathbb U)$. 

Recalling that $\cup_{g\in G} g(K) = X(\mathbb U)$, we choose $g\in G$ with $g(x)\in K$. We now show $S':= \{g sg^{-1} : s \in S\} \subset A$. Indeed, for each $s\in S$, we have 
\begin{align*}
d(p, g s g^{-1}(p)) &\le d(p, g(x))+d(g(x), g s(x))+d(g s(x), g sg^{-1}(p)) \\
&\le d(p, g (x))+ d(x, s(x)) + d(g(x), p) \le 2a + M.    
\end{align*}

Since $S'  \subset A$ generates $H':=g Hg^{-1}\in \mathbb H$, by the definition of $n_0$, there is some integer  $1 \le k \le n_0$ such that $d_{S'}(1, g_{H'})=k$. Thus,    $$g_{H'} = (g s_1g^{-1})\cdots(g s_k g^{-1}) = g(s_1\cdots s_k)g^{-1}$$ 
for $s_i \in S \cup S^{-1}$.
Now by
 triangle inequality, we have \begin{align*}
d(g^{-1}g_{H'}g(x), x) & = d(g_{H'}g(x), g(x)) \\
&\ge d(g_{H'}(p), p) - d(g_{H'}(p), g_{H'}g(x))- d(g(x), p)\\
&= d(g_{H'}(p), p) - d(p, g(x)) - d(g(x), p) \\
&> M + 2a- a - a = M.    
\end{align*}
Since $g^{-1}g_{H'}g = s_1 \cdots s_k \in S^{\le n_0}$, it follows that $\ell_x(S^{\le n_0}) > M$. 
\end{proof}

The following result improves   Proposition 3.2 of \cite{Koubi}.
\begin{lem}\label{ShortHypLem}
Let $X$ be a $\delta$-hyperbolic geodesic metric space, and $H$ a group of isometries of $X$ with a finite symmetric generating set $S$. If $\ell_X(S) > 28\delta$, then $H$ contains a loxodromic element $b\in S^{\le 2}$. Moreover, there exists a constant $C=C(\delta)>0$ such that   $$d(o,bo)\ge \ell_X(S)-C$$ for some point $o\in X$.

\end{lem}

\begin{proof}

Let $o\in X$ such that $\ell_{X}(S)+\delta>\ell_o(S)\ge \ell_{X}(S)$. Set $L_0=4\delta$ and then  $\ell_o(S)>7L_0$. Denote by $S_0$ the (non-empty) set of elements $s\in S$ so that $$d(o, so)\ge \ell_o(S)-2L_0-\delta.$$

Let   $t\in S$ such that $\ell_o(S)=d(o, to)$, and $m\in [o, to]$ so that $d(o, m) =L_0$. 

The main observation is as follows.
\begin{claim}
There exists an isometry $s\in S_0$ such that  $s$ is either loxodromic with $\langle o, s^2o\rangle_{so}\le L_0$  or satisfies 
$$
\max\{\langle t o, s o \rangle_o,\langle t^{-1} o, s^{-1} o \rangle_o\}\le L_0.  
$$ 
\end{claim}
\begin{proof}[Proof of the Claim]
Assume to the contrary that for all $s\in S_0$, we have 
\begin{equation}\label{MaxiumEQ}
\max\{\langle t o, s o \rangle_o,\langle t^{-1} o, s^{-1} o \rangle_o \}> L_0.      
\end{equation}
Moreover, each $s\in S_0$ is either non-loxodromic or loxodromic   with $\langle o, s^2o\rangle_{so}> L_0$. 
If $s\in S_0$ is non-loxodromic, by Lemma \ref{LoxoCriterionLem}, we have $$\langle o, s^2o\rangle_{so}\ge d(o,so)/2-3\delta\ge (\ell_o(S)-2L_0-\delta)/2-3\delta\ge L_0.$$ Hence, for each $s\in S_0$, we have $\langle o, s^2o\rangle_{so}\ge L_0$. In particular, $\langle t^{-1}o, to\rangle_{o}\ge L_0$.

By (\ref{MaxiumEQ}), assume that $\langle t^{\star} o, s^{\star} o \rangle_o>L_0$ for $\star\in \{1, -1\}$. Let $m_1, m_2\in [o, s^\star o]$ for $s\in S_0$ so that $d(o, m_1)=d(s^\star o, m_2)= L_0$. By hyperbolicity,  $$\langle s^\star o, to\rangle_o\ge \min\{\langle s^\star o, t^\star o\rangle_o, \langle t^{-1}o, to\rangle_{o}\}-\delta \ge L_0-\delta$$ which by the $\delta$-thin triangle property implies $d(m, m_1)\le 3\delta$. Using again $\delta$-thin triangle with $\langle o, (s^\star)^2o\rangle_{s^\star o} \ge L_0$, we obtain that $d(m_2, s^\star m_1)\le \delta$.

We shall derive  $\ell_m(S)<\ell_X(S)$, which is a  contradiction. Indeed, for each $s\in S_0,$
\begin{align*}
d(m, s^\star m)&\le  2 d(m,m_1)+d(m_1, s^\star m_1)\le 7\delta+d(m_1,m_2)\\
&\le 7\delta+d(o,s^\star o)-2L_0 \\
& \le \ell_o(S)-2L_0+7\delta \le \ell_o(S)-\delta.  
\end{align*} 
The definition of $s\in S\setminus S_0$ gives $d(o, so)\le \ell_o(S)-2L_0-\delta$ and thus  $$d(m, sm)\le 2d(o, m)+d(o, so)\le 2L_0+d(o,so)< \ell_o(S)-\delta$$
We obtained the contradiction $\ell_m(S)\le \ell_o(S)-\delta< \ell_X(S)$. The proof of the claim is now complete.  
\end{proof}

By the above claim,  there exists $s\in S_0$ such that either
$$\langle s o, s^{-1} o \rangle_o+\delta \le L_0+\delta \le \frac{1}{4} d(o,so)$$
or  
$$
\max\{\langle t o, s o \rangle_o, \langle t^{-1} o, s^{-1} o \rangle_o\} +\delta \le L_0+\delta \le \frac{1}{4} \min\{d(o,so), d(o, to)\} 
$$
The proof is then completed by   Lemma \ref{LoxoCriterion2Lem}. 
\end{proof}

\section{Short loxodromic elements with large injectivity}\label{LargeInjSection}

Let  $\lambda_0=\lambda_0(\delta), c_0=c_0(\delta), C_0=C_0(\delta)$ be given by Lemma \ref{ShortHypLem}. Let $n_0=n_0(28\delta)$ be given by Lemma \ref{LargeDiamLem} so that $\ell_X(S^{\le n_0})> 28\delta$.  Thus, $S$ contains a  loxodromic element $b\in H$ and there exist a point $o\in X$ such that the path 
\begin{align}\label{alphaDef}
\alpha=\bigcup_{n\in \mathbb Z} b^n[o,bo]    
\end{align} is a $(\lambda_0, c_0)$-quasi-axis for $b$. 

By hyperbolicity, any quadrilateral  with $(\lambda_0, c_0)$-quasi-geodesic sides is $C$-thin for some $C=C(\lambda_0, c_0, \delta)>\max\{C_0, \delta\}$: any side is contained in the $C$-neighborhood of the other three sides. 

Since $\sharp (S^{2n_0}) \ge \sharp S$, it suffices to prove Theorem \ref{mainthm2} assuming the generating set $S$ with $$\ell_X(S)>\max\{28\delta,2C\}.$$ 
By Lemma \ref{ShortHypLem},  we have 
\begin{align}\label{distboEQ}
d(o, bo)\ge \ell_X(S)-C_0\ge C.    
\end{align}  

Since $H$ is not virtually cyclic, $S$ contains an element $f$ such that $f\notin E(b)$. Indeed, if not, any $f\in S$ would fix the set of fixed points of $b$ so it follows from $H=\langle S\rangle$   that the limit set of $H$ consists of two points. By the subgroup  classification    in (the convergence action of) $G$, we obtain that $H$ would be virtually cyclic. This is a contradiction. 
\subsection{Loxodromic elements raising to power} 
In the remainder of this section, we assume that $f\in S\setminus E(b)$. Consider the element $h:=fb^{n}$ for $n\ge 0$.

\begin{figure}[htb] 
\centering \scalebox{0.6}{
\includegraphics{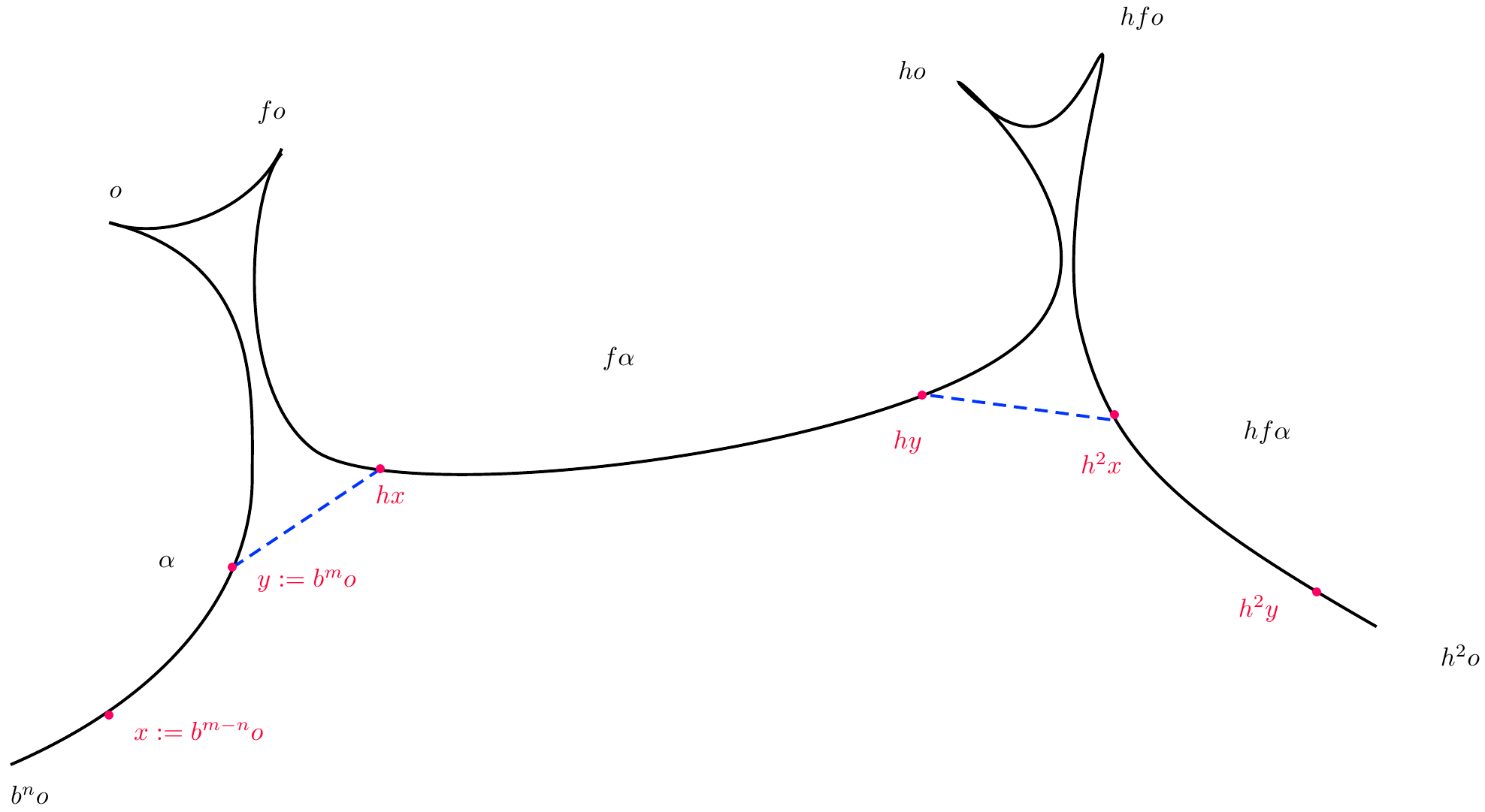} 
} \caption{Truncate the quadrilaterals $y,o,fo,hx$ and $hy, ho, hfo, h^2x$} \label{figure1}
\end{figure}

Consider the points $x=b^{m-n}o, y=b^{-m}o$ for $0\le m\le n$ on  $\alpha$, and then $hx=fb^{m}o\in f\alpha$. Consider two quasi-geodesics $\alpha$ and $f\alpha$ connected by a geodesic $[o,fo]$. To get a quasi-axis of $h$, we shall use the next lemma to truncate the part containing $[o,fo]$ of $\alpha$ and $f\alpha$ at the points $y$ and $hx$.
\begin{lem}\label{SmallProductLem}
For any $\theta>0$, there exists $n=n(\theta, \delta), m=m(\theta, \delta)>0$ so that $h=fb^n$ for any $n>n(\theta, \delta)$ enjoys the following property. 

If $\theta d(o,bo) \ge d(o,fo)$, then  
$\langle x,hx\rangle_{y}, \langle y, hy\rangle_{hx} \le C$ and $d(y, hx)\ge  \theta d(o,bo).$
\end{lem}
\begin{proof}
Consider the quadrilateral formed by the subpaths $[x, o]_\alpha$, $[o,fo]$, $f[o,b^{m}o]_\alpha$, and the geodesic $[x, hx]$ as depicted in Figure \ref{figure1}. 

With the inequality (\ref{distboEQ}), the $(\lambda_0,c_0)$-quasi-geodesicity of $\alpha$ in (\ref{alphaDef}) gives   
\begin{equation}\label{ndistboEQ}
\begin{array}{rl}
\forall n\ge 1, \; n\cdot d(o,bo)\ge d(o, b^{n}o)&\ge \lambda_0^{-1} n\cdot d(o,bo)-c_0\\
&\ge  \lambda_0^{-1} n \cdot d(o, fo) -(c_0+nC_0/\lambda_0). 
\end{array}
\end{equation}

Set $N=N(\lambda_0,c_0,\delta, \theta)$ given by Lemma \ref{LinearBILem}. By (\ref{ndistboEQ}), we can choose the least integer $m=m(\lambda_0,c_0,\delta)> N$ and then the least $n=n(\lambda_0,c_0, \theta, \delta)$   so that for any $n\ge n(\theta, \delta)$, \begin{equation}\label{n1DefnEQ}
d(o,b^{m}o), d(o,b^{n-2m}o)\ge \max\{N d(o,bo), 10C\}.     
\end{equation}

If denote $R:=\theta d(o,bo) \ge d(o,fo)$, then 
$$
d(y, f\alpha)> R.$$
Indeed, assume to the contrary that $d(y, f\alpha)\le R$. Since  $d(o,f\alpha)\le d(o, fo)\le R$, by Lemma \ref{LinearBILem}, the diameter of   $\alpha\cap N_R(f\alpha)$  is at most  $Nd(o,bo)$. However, $\alpha\cap N_R(f\alpha)$ contains two points $y=b^{-m}o,o$ with distance at least $Nd(o,bo)$ by Eq. \ref{n1DefnEQ}. This is a contradiction. 
Thus,  $d(y, f\alpha)> R$ is proved. The $C$-thin quadrilateral property then implies  $d(y, [x, hx])\le C$, so  we obtain $\langle x, hx\rangle_y \le d(y, [x, hx])\le C$. 

By symmetry, we can run the above argument   for the quadrilateral with vertices $y, o, fo, hy$ and obtain $\langle y, hy\rangle_{hx} \le C$.

Note that $d(y, hx)\ge d(y, f\alpha)\ge R=\theta d(o,bo)$. The proof is complete.
\end{proof}

\begin{lem}\label{HighPowerElemLem}
There exist constants $n_1=n_1(\delta), m_1=m_1(\delta),   \lambda=\lambda(\delta), c=c(\delta)>0$ such that for any $n\ge n_1$, the element $h:=fb^{n}$ is loxodromic with a $(\lambda, c)$-quasi-axis $\beta$ defined  as follows
\begin{align}\label{BetaPathEQ}
\beta:=\bigcup_{i\in \mathbb Z}    h^i\left([x,y]_\alpha[y,hx]\right) 
\end{align} 
where $x=b^{m_1-n}o, y=b^{-m_1}o$.
\end{lem}
\begin{proof}
The inequality (\ref{ndistboEQ}) implies that the following   constant $\theta$ depends only on $\delta$: $$\theta:=\max\left\{1, \frac{d(o,fo)}{d(o,bo)}, \frac{10C}{d(o,bo)}\right\}.$$ 
If $m_1:=m(\delta, \theta)$ is given by Lemma \ref{SmallProductLem}, then for $n\ge m_1$,  we have $$\max\{\langle x,hx\rangle_{y}, \langle y, hy\rangle_x \}\le C.$$

Note that $d(x,y)=d(o, b^{n-2m_1}o)\ge 10C$ by  Eq. (\ref{n1DefnEQ}) and $d(y, hx)\ge \theta d(o,bo) \ge 10C$. Therefore,
$$
\langle x,hx\rangle_{y}+ \langle y, hy\rangle_{hx}\le \frac{1}{4} d(y, hx) -\delta.
$$
so the assumption of  Corollary \ref{LocalGeodCor} is verified for the sequence of points $$\cdots, h^{-i}x,h^{-i}y,\cdots, x, y, hx, hy, \cdots, h^j x, h^j y,\cdots$$ Hence, there exist $\lambda=\lambda(\delta), c=c(\delta)>0$ such that $\beta$ is a $(\lambda, c)$-quasi-geodesic. This proves that  $h$ is loxodromic.
\end{proof}

\subsection{Large injectivity} 
The crucial property in   constructing free subgroups is the following property of a loxodromic isometry $h=fb^n$. Recall that the point $o\in X$ is provided by Lemma \ref{ShortHypLem} so that the inequality (\ref{distboEQ}) holds.
\begin{defn}\label{LoxLargeInjDef}
A loxodromic element $h$ has  \textit{injective radius} $L>0$ if $E(h)$  contains a finite subgroup $F$ with $[F: E^\star(h)]\le 2$   so that $E=\langle h\rangle F$ and for any $g\in E(h)\setminus F$, we have $\ell_X(g)>L\cdot d(o,bo)$.
\end{defn}
 
Let $N_0>0$ be given by Lemma \ref{UniFiniteBoundLem} so that $\sharp E^\star(h)\le N_0$, and $D=D(\lambda, c, \delta)$  be       given by Lemma \ref{ShortEllipticLem}.
\begin{lem}\label{TwoPolesLem}
For any $L>0$ there exists $n_2=n_2(L,\delta)\ge n_1$  such that the loxodromic  element $h=fb^n$ for $n\ge n_2$ has injective radius $L$. Precisely,   
\begin{enumerate}
    \item 
    $\sharp F \le 2 N_0$ and $\ell_z(F)\le 2D$ for any $z\in \beta$.
    \item 
    For any $g\in E(h)\setminus F$, we have $\ell_X(g)>L\cdot d(o,bo).$
    \item
    for any $g\in  E(h)\setminus F$, there exist    $i\in \mathbb Z$ and $t\in F$  such that $g=h^i t$.
\end{enumerate}  
\end{lem}

\begin{proof}
We keep the same notation as in the   proofs of Lemma \ref{SmallProductLem} and Lemma \ref{HighPowerElemLem}. For any $h=fb^n$ with $n\ge n_1$, the path $\beta$ in   (\ref{BetaPathEQ}) is a $(\lambda,c)$-quasi-geodesic, where the constants $n_1, \lambda, c>0$ depend only on $\delta$. 

Let $m_1=m_1(\delta)$ given by Lemma \ref{HighPowerElemLem}. Denote $x=b^{m_1-n}o, y=b^{-m_1}o$. Then 
\begin{equation}\label{CpDistEQ}
\begin{array}{rl}
d(y, hx)&\le d(o,fo)+2d(o, b^{m_1}o), \\
d(x, y)&=d(o, b^{n-2m_1}o).    
\end{array}   
\end{equation}

By Lemma \ref{QuasiAxisLengthLem}, it suffices to prove the statement  (2) by placing the basepoint  $z$ to the point $y=b^{-m_1}o$ at $\beta$. Since $\beta$  is a $(\lambda,c)$-quasi-geodesic, by increasing  $n_1=n_1(\delta)$, we can assume  
\begin{equation}\label{LargeGapEQ}
\forall i\ne 0\in \mathbb Z,\; d(y, h^iy)> 2D+L\cdot d(o,bo).
\end{equation}

We first consider  elements $g\in E^+(h)$ and prove the corresponding statements (2-3).

Denote $\beta_0=[x,y]_\alpha[y,hx]$  the fundamental domain for the action of $\langle h \rangle$ on $\beta$. By hyperbolicity,   the finite Hausdorff distance $d_H(\beta,  g\beta)<\infty$ implies  a uniform constant $R=R(\delta)>0$ so that $d_H(\beta,  g\beta)\le R$.   

By (\ref{distboEQ}), the constant $\theta$ defined as follows depends on $\delta$ only: $$\theta:=\displaystyle \frac{R}{d(o,bo)}\le R/C.$$    Let $N=N(\lambda,c,\delta, \theta)$ be given by Lemma \ref{LinearBILem}.

By Eq. (\ref{CpDistEQ}), the least integer $n_2\ge \max\{N, n_1\}$  such that for any $h=fb^n$ with $n>n_2$, 
\begin{equation}\label{LongdistxyEQ}
d(x, y) \ge   d(y, hx)+2n_2d(o,bo)    
\end{equation}
depends only on $\delta$.
 
Note that $\beta$ is contained in the union   $\bigcup_{i\in \mathbb Z} h^i \alpha$ and $\bigcup_{i\in \mathbb Z} h^i [y, hx]$. By (\ref{LongdistxyEQ}), the path $g[x,y]_\alpha$ contains a subpath $\alpha_0$ of diameter at least $n_2d(o,bo)$ which is contained in  the $R$-neighborhood $h^i\alpha$. Then there exist  a subpath $\alpha_1$ of $\alpha$  such that  $d_H(g\alpha_0, h^i\alpha_1)\le R$.  Since $n_2>N$, Lemma \ref{LinearBILem} implies $g^{-1} h^i\in E(b)$. 

We claim that $t:=g^{-1}h^i\in E(h)\cap  E(b)$ is of finite order. If not,  then $E(h)\cap E(b)$ is an infinite subgroup. Thus, $E(h)\cap E(b)$ act co-compactly on   the quasi-axis of both $h$ and $b$, so the $(\lambda,c)$-quasi-axis  of $b$ is preserved by $h$ up to finite Hausdorff distance. Hence, we obtain $h \in E(b)$ and then $f\in E(b)$. This is a contradiction.

Therefore, the finite order element $t\in E(b)$ preserves a $(\lambda, c)$-quasi-axis $\beta$ of $h$.  By Lemma \ref{ShortEllipticLem}, we have {$d(z, tz)<D$}. So far, we have verified the assertions (2-3) for $g\in E^+(h)$.

To complete the proof, it remains to consider elements $g\in E^-(h)$. If $d(y,gy)>D$ for all $g\in E^-(h)$, then we are done by setting $F:=E^\star(h)$ and $\ell_X(F)\le D$ by Lemma \ref{ShortEllipticLem}. Otherwise, let $r\in E^-(h)$ so that $d(y,ry)\le D$. Thus,  $F:=\langle E^\star(h), r\rangle$ has order at most $2N_0$ and $\ell_y(F)\le \ell_y(E^\star(h))+d(y, ry)\le 2D$.  

Since $E^+(h)$ is of index $2$ in $E(h)$, we write $g=h^i t v$ for some $t\in E^\star (h)$.   If $i\ne 0$, one deduce from (\ref{LargeGapEQ}) that
$$
d(y, gy)\ge d(y, h^iy)-d(y, ty)-d(y, ry) \ge Ld(o,bo). 
$$
The result  is proved. \end{proof}

\section{Proof of Theorem \ref{mainthm2}}\label{MainProofSection}
 
We resume the constants $\lambda_0, c_0, C_0, C$ depending only on $\delta$ at in Section \ref{LargeInjSection} and the results obtained there under the assumption $\ell_X(S)>\max\{28\delta,2C\}$. Then $b\in S^{\le 2}$ is a loxodromic element given by Lemma \ref{ShortHypLem}, and $f\in S\setminus E(b)$ exists due to the fact that $H$ is a non-elementary subgroup.

Let $m_1=m_1(\delta)\ge 2$   given by Lemma \ref{HighPowerElemLem} and make the reference point at $y=b^{-m_1}o$ on the quasi-axis $\beta$ in (\ref{BetaPathEQ}). Note that $d(o,bo)=d(y,by)$. 

Set   $L:=4(m_1+1)\ge 10$. By (\ref{distboEQ}), we have $d(o,so)\le C_0+d(o,bo)$  for any $s\in S$, and thus 
\begin{equation}\label{SmallsDistEQ1}
d(y, sy)\le (2m_1+1)d(o,bo)+C_0< \frac{L}{2}d(y,by)   
\end{equation}
which yields for  any $c:=s^{-1}s'$ with $s\ne s'\in S$,
\begin{equation}\label{SmallsDistEQ}
d(y, cy)< L\cdot d(y, by)=L\cdot d(o,bo).
\end{equation}

Let  $n_2=n_2(L,\delta)>n_1$ and $F$  be the finite subgroup in $E(h)$    given by Lemma \ref{TwoPolesLem}.  The following result holds for   any integer $n\ge n_2$ and   $h=fb^n$.

\begin{lem}
Choose  a largest subset $S_0$ of $S$ such that $sF \ne   s'F$ for any $s\ne s'$. 
Then  for any $s\ne s'\in S_0$, $s^{-1}s'\notin E(h)$.
\end{lem}
\begin{proof} 
By Lemma \ref{TwoPolesLem}, the inequality (\ref{SmallsDistEQ}) implies  that $s^{-1}s'$ must be contained in $F$ so $sF=s'F$. This contradicts to the choice of $S_0$ consisting of different left $F$-coset representatives.  
\end{proof}

Choose the least integer $n_3\ge n_2$ such that the last inequality in (\ref{yhyEQ2}) holds for any $n\ge n_3$: 
\begin{align}\label{yhyEQ}
&d(y, hy)=d(y, fb^ny)=d(b^{m_1}o, fb^{n+m_1}o)\\\label{yhyEQ2}
\ge& d(o, b^no)-2d(o, b^{m_1}o)-d(o,fo)> Ld(o,bo).    
\end{align}

\paragraph{\textbf{Construct the free bases}} 
Let us now fix   $h=fb^{n_3}$ throughout the proof. Let $F$ be the finite subgroup in $E(h)$   by Lemma \ref{TwoPolesLem}. Let $S_0$ be a largest subset of $S$ such that $sF \ne   s'F$ for any $s\ne s'$.

For $\theta=1$, let $m_2=m(1,\delta), k=n(1,\delta)$  given by Lemma \ref{SmallProductLem}.

We define the free base as follows: $$T=\{th^kt^{-1}: t\in S_0\}.$$ 
If set $\kappa:=2+k(n_3+1)$, then  $d_S(1, th^kt^{-1})\le \kappa$ and $T\subset S^{\le \kappa}$.

The goal is the following.
\begin{lem}
The set $T$ generates a free subgroup of rank $\sharp T$ in $H$.
\end{lem}
\begin{proof}
Let $W$ be a non-empty reduced word over $T\cup T^{-1}$ written as follows
\begin{align*}
W&=(s_1\cdot h^{i_1k}\cdot s_1^{-1})( s_2 \cdot h^{i_2k}\cdot s_2^{-1})\cdots ( s_l \cdot h^{i_lk}\cdot s_l^{-1}) \\
&=s_1\cdot \left( h^{i_1k} \cdot c_1\cdot  h^{i_2k} \cdot c_2  \cdots \cdot c_{l-1} h^{i_lk}\right)\cdot s_l^{-1}
\end{align*}

First of all, let $\beta_j$  be the 
subpath of $\beta$ starting from $y$ to $h^{i_jk}y$ consisting of $
i_j\cdot k$ copies of $[y, hy]_\beta$. Let $p_j=[y, c_jy]$ be a geodesic  labeled by $c_j$.

We choose $z_j, w_j$ on $\beta_j$ so that the initial subpath of $\beta_j$ until $z_j$ contains exactly $m_2$ copies of $[y, hy]_\beta$, and the terminal path starting at $w_j$  contains exactly $m_2$ copies of $[y, hy]_\beta$.  To be precise, set $z_j=h^{m_2k}y, w_j=h^{i_j-m_2}y$.

Furthermore if $j=1$, we let $z_1$ be the initial point of $\beta_1$; if   $j=l$, let $w_l$ be the initial point of $\beta_l$.

We now properly translate $\beta_j$ and $p_j$ for $1\le j\le l$ so that $\beta_1$ originates at $y$, and then the terminal points of $\beta_j$ followed by the initial points of $p_j$ in a manner produces the following concatenated  path: $$\gamma=\beta_1\cdot p_1 \cdot\beta_2\cdot p_2\cdot \beta_3 \cdots p_{l-1}\cdot \beta_l.$$ 
(We refer the reader to Figure \ref{figure1} for similar illustration of cutting out quadrilaterals, where $x, y, hx, hy, h^2x, h^2y$ should be marked as $z_1,w_1, z_2, w_2,z_3, w_3$ etc.)

By abuse of language, the translated points of $z_j,w_j$ on $\beta_j$ are still denoted by $z_j,w_j$, so we have plotted  a sequence of points $z_1, w_1, z_2, w_2, \cdots, z_l, w_l$ on $\gamma$. By the choice of $z_1, w_l$, the path $\gamma$ starts at $z_1$ and ends at $w_l$, labeled by the word $s_1^{-1} Ws_l$.

The key construction is then to cut  quadrilaterals off $\gamma$ along $[w_j, z_{j+1}]$ and verify that $\{z_1, w_1, z_2, w_2, \cdots, z_l, w_l\}$  is a quasi-geodesic.

To truncate the quadrilaterals, we apply Lemma \ref{SmallProductLem} to $\beta_j, c_j, \beta_{j+1}, c_{j+1}$ in order for $1\le j\le l$. For concreteness, set $j=1$.

Noting $d(y,hy)>Ld(y,cy)$ by (\ref{yhyEQ2}) and (\ref{SmallsDistEQ}), Lemma \ref{SmallProductLem} gives $$\langle z_1, z_2\rangle_{w_2}, \langle w_1,  w_2\rangle_{z_2} \le C$$ and together with Lemma \ref{TwoPolesLem}.(2), $$d(w_1,  z_2)\ge   \theta d(y,hy)\ge Ld(y, by).$$
By the inequality (\ref{distboEQ}) that $d(y, by)\ge C\ge \delta$,  we  thus derive 
\begin{align}\label{ShortProdEQ1}
\langle z_1, z_2\rangle_{w_2}, \langle w_1,  w_2\rangle_{z_2} \le d(w_1, z_2)/4-\delta    
\end{align}

Similarly, since $d(z_2,w_2)=d(y, h^{i_2k-2m_2}y)\ge  Ld(y, by) \ge 10C$, we have 
\begin{align}\label{ShortProdEQ2}
\langle w_1, w_2\rangle_{z_2}, \langle z_2,  z_3\rangle_{w_2} \le d(z_2, w_2)/4-\delta
\end{align}

In conclusion,  the inequalities (\ref{ShortProdEQ1}) and (\ref{ShortProdEQ2}) verifying the assumption of Corollary \ref{LocalGeodCor} hold for every four consecutive points in $z_1, w_1, z_2, w_2, \cdots, z_l, w_l$. Thus,  $$
d(z_1, w_l)\ge \frac{1}{2} \sum_{1\le j \le l} d(z_j, w_j) \ge Ld(y,by)
$$ 
By (\ref{SmallsDistEQ1}), we have $d(y, s_1y)+d(y, s_ly)< Ld(y,by)$. 
Thus, $d(o, Wo) =d(z_1,w_1)-d(y, s_1y)-d(y, s_ly)>0$. Hence, any non-empty reduced word $W$ is mapped a non-trivial isometry, so $T$ generates a free subgroup of rank $\sharp T$.  
\end{proof}

We now finish the proof of Theorem \ref{mainthm2}. Summarizing the above discussion, for each generating set $S$ of $H$, we constructed a finite set $T\subset S^{\le \kappa}$ satisfying $$\sharp T \ge \frac{1}{2N_0}\sharp S$$ 
so that $\langle T\rangle $ is a free group of rank $\sharp T$. Thus, $$\sharp (S^{\le n\kappa})\ge  (2\sharp T-1)^n\ge (\frac{\sharp S-N_0}{N_0})^n$$ and there exists $c_0>0$ such that $\omega(H,S)\ge c_0$ for any finite symmetric set $S$. 

Choose the least integer  $M=M(N_0)>0$ such that    $\sharp S/N_0\ge 1+\sqrt{\sharp S}$ for any $\log\sharp S>M$. In this case, we thus obtain $$\omega(H, S)\ge    \frac{1}{2\kappa}{\log (\sharp S)}.$$
Otherwise, $\log\sharp S\le M$, we have $$\omega(H, S)\ge c_0\ge \frac{c_0}{M}{\log (\sharp S)}.$$ The proof of Theorem \ref{mainthm} is finished.

\bibliographystyle{amsplain}   
\bibliography{bibliography}
\end{document}